\pdfoutput=1

\documentclass{article}
\usepackage[utf8]{inputenc}

\title{Large Rank Simple Bundles of All Homological Dimensions}
\author{Kaiying Hou \\ {khou@college.harvard.com}}
\date{\today}
\usepackage[numbers]{natbib}
\usepackage{graphicx}
\usepackage{multicol}
\usepackage[margin=1in]{geometry}
\usepackage{amsmath}
\usepackage{amsfonts}
\usepackage{amssymb}
\usepackage{mathrsfs}
\usepackage{tikz-cd}
\usepackage[shortlabels]{enumitem}

\newcommand{\Pp}{\mathbb{P}}

\newcommand{\Oo}{\mathcal{O}}

\newcommand{\End}{\text{End}}
\newcommand{\coker}{\operatorname{coker}}
\newcommand{\Hom}{\operatorname{Hom}}

\DeclareMathOperator{\im}{im}
\DeclareMathOperator{\Gr}{Gr}

\usepackage{amsmath}
\usepackage{setspace}

\usepackage{amsthm}
\newtheorem{theorem}{Theorem}[section]
\newtheorem{proposition}[theorem]{Proposition}
\newtheorem{corollary}[theorem]{Corollary}

\newtheorem{definition}[theorem]{Definition}
\newtheorem{example}[theorem]{Example}

\newtheorem{remark}[theorem]{Remark}
\newtheorem{lemma}[theorem]{Lemma}

\begin{document}

\maketitle

\begin{abstract}
   For $n\geq 3$ and $r\geq n$, we show that there are rank-$r$ vector bundles on $\Pp^n$ with arbitrary homological dimension. 
   We apply the Bernstein-Gel'fand-Gel'fand correspondence to translate the vector bundle question into a problem on modules over the exterior algebra.
   Then, we use linear algebra to construct the desired modules.
\end{abstract}

\section{Introduction}
The classification of algebraic vector bundles on projective spaces has been an active field, with interesting open problems such as Hartshorne's conjecture on low rank non splitting bundles \cite[p.1]{Hartshorne_conjecture}.
As with other areas of mathematics, invariants that partition the objects of study into subgroups are helpful for classification.
One invariant for vector bundles is the homological dimension introduced by Bohnhorst and Spindler in \cite{spindler_stability}.
\begin{definition}
    A resolution of a vector bundle $\mathcal{F}$ on projective space is a chain complex of sheaves
    $$\mathcal{C}^{\bullet} : \,\,\,\cdots \longrightarrow \mathcal{C}^{-1} \longrightarrow \mathcal{C}^0 \longrightarrow 0$$ such that each $\mathcal{C}^i$ is a direct sum of line bundles, $H^i(\mathcal{C}) = 0$ for $i\neq 0$ and $H^0(\mathcal{C}) = \mathcal{F}$.
    The homological dimension of $\mathcal{F}$, denoted by $hd(\mathcal{F})$ is the shortest length a resolution of $\mathcal{F}$ can have.
\end{definition}
\begin{remark}
    We are using cohomological grading for convenience.
\end{remark}
The homological dimension appears in the study of low rank non splitting bundles due to Corollary 1.7 of \cite{spindler_stability}, which relates the rank $\text{rk}(\mathcal{F})$ with $\text{hd}(\mathcal{F})$ for any non splitting bundle $\mathcal{F}$ on $\Pp^n$ through the inequality
$$
\text{rk}(\mathcal{F}) \geq n+1 - \text{hd}(\mathcal{F}).
$$
Using Horrock's splitting criterion, one can show that for $\mathcal{F}$ on $\Pp^n$, $\text{hd}(F) = 0,1,...,$ or $n-1$, where $\text{hd}(\mathcal{F}) = 0$ is equivalent to $\mathcal{F}$ being a direct sum of line bundles \cite{spindler_stability}.
In 2012, Jardim and Prata constructed rank-$n$ simple vector bundles on $\Pp^n$ of homological dimensions $1,2,...,$ and $n-1$, proving that for rank-$n$ bundles, all homological dimensions are possible \cite[Theorem 1.3]{Jardim}.
They proved this result by performing induction on the homological dimension using  Theorem 4.3 in Brambilla's \cite{brambilla2007cokernel}.
However, it remained unclear whether the same was true for ranks other than $n$.

One useful technique for studying sheaves on projective spaces is the Bernstein-Gel'fand-Gel'fand correspondence introduced in \cite{Bernshtein1978AlgebraicBO}, which given a $\bigwedge V-$module $P$ where $V$ is a finite dimensional vector space, constructs a chain complex $L(P)$ of free $\text{Sym} V^*$-modules.
Sheafification produces a complex $\tilde{L}(P)$ of vector bundles on $\Pp^n = \text{Proj}(\text{Sym} V^*)$.
Therefore, complexes of sheaves become related to modules over $\bigwedge V$.
For instance, Eisenbud, Fløystad, and Schreyer constructed the Beilinson monad using the correspondence \cite{Eisenbud_floystad}.
Moreover, by restricting to certain kinds of $\bigwedge V$-modules, one can ensure that $\tilde{L}(P)$ has cohomologies that are vector bundles.
Coandă and Trautmann considered complexes of $\bigwedge V$ modules that, through the BGG correspondence, gave rise to stable vector bundles \cite{Coanda}. Alternatively, results on vector bundles can produce insights on modules over $\bigwedge V$: Popa and Lazarsfeld discovered Hodge number inequalities by passing cohomology rings through the BGG correspondence  \cite{Popa_bgg}.
In this paper, we are able to generalize the result of \cite{Jardim} to arbitrary ranks larger than or equal to $n$ using the BGG correspondence.
\begin{theorem}\label{thm:arb_hd}
Let $n\geq 3$ and let $k$ be an algebraically closed field.
For $l = 1,2,...,n-1$ and any $r\geq n$, there exists a simple vector bundle of rank $r$ and homological dimension $l$ on $\Pp^n_k$, the $n$-dimensional projective space over $k$.
\end{theorem}

This paper is organized as follows. Section \ref{sec:preliminaries} introduces the necessary definitions.
Section \ref{sec:lemmas} proves a lemma about linear resolutions that reduces the simplicity of vector bundles to the simplicity of modules over $\bigwedge V$.
Section \ref{sec:construction} gives an explicit construction of the vector bundles used in Theorem \ref{thm:arb_hd}.

\section{Preliminaries} \label{sec:preliminaries}
\subsection{BGG Complex}
Let $k$ be an algebraically closed field.
Let $V$ be an $n+1$-dimensional $k$-vector space.
Let $E = \bigwedge V$ be its exterior algebra.
In this paper, for convenience, we assume $E$ is graded positively, i.e., the degree of any $v\in V$ is $1$.
Let $P$ be a graded left $E$-module.
Let $\{e_0,...,e_n\}\subset V$ be an arbitrary basis and $\{x_0,...,x_n\}\subset V^*$ be the corresponding dual basis.
Let $S = \text{Sym} V^* = k[x_0,...,x_n]$.
The following definition of the BGG complex is from \cite{Eisenbud_syzygies}.
\begin{definition}
    Let $P= \oplus_{i\in\mathbb{Z}} P_i$ be a graded $E$-module.
    Then, the BGG complex $L(P)$ is given by
    $$ \cdots \longrightarrow S[i]\otimes_kP_i \longrightarrow S[i+1]\otimes_kP_{i+1}\longrightarrow \cdots$$
    where the differential is defined by $$1\otimes p \mapsto \sum_{i=0}^n x_i\otimes e_ip.$$
\end{definition}
We call a chain complex $$\cdots\longrightarrow M^i\longrightarrow M^{i+1}\longrightarrow\cdots$$ of $S$-modules a linear free complex if each $M^i$ appearing in it is free and the generators of $M^i$ have degree $-i$.
The BGG correspondence then says the following.
\begin{proposition} \label{prop:equiv_cat}
    $L$ is an equivalence of categories from graded $E$-modules to linear free complexes over $S$.
\end{proposition}
This version of the correspondence is presented in \cite{Eisenbud_syzygies}.
It is worth noting that the BGG correspondence can also be stated as an equivalence of categories between the derived category of $S$-modules and the derived category of $E$-modules as shown in Corollary 2.7 of \cite{Eisenbud_floystad}.
However, we will not be needing this derived equivalence in our paper.

\subsection{BGG-Sheaves and Faithful Modules}
In the rest of the paper, we assume that $0 = \min_i(P_i\neq 0)$, i.e., the module ``starts" at the $0$th graded piece.
Also, the $E$-modules considered will all be finitely generated.
We can sheafify $L(P)$ to produce the complex $\tilde{L}(P)$ given by 
$$\cdots\longrightarrow P_i \otimes_{k} \mathcal{O}_{\Pp^n}(i) \xrightarrow[]{\,d_i\,} P_{i+1} \otimes_{k} \mathcal{O}_{\Pp^n}(i+1) \longrightarrow \cdots $$
Following \cite{Popa_bgg}, we call the cohomology at the last nonzero term the BGG-sheaf.
\begin{definition}
    The BGG-sheaf refers to $H^c(\tilde{L}(P))$ where $c=\text{max}_i(P_i\neq 0)$.
\end{definition}
One can show using Castelnuovo-Mumford regularity that every vector bundle can be realized as the BGG-sheaf of some module $P$.
For our application on homological dimensions, we care about the case when the BGG-sheaf is a vector bundle and $\tilde{L}(P)$ provides a resolution for it. 
To understand what this means for the module $P$, we will use a basis-free description of the differential maps $d_i$s appearing in $\tilde{L}(P)$.
Let $$d_i(-i-1)
:P_i \otimes_{k} \mathcal{O}_{\Pp^n}(-1) \to P_{i+1} \otimes_{k} \mathcal{O}_{\Pp^n}$$ be the map obtained by twisting $d_i$.
Over $[v] \in \Pp^n$, the fiber has the form $$\left(P_i \otimes_{k} \mathcal{O}_{\Pp^n}(-1)\right)|_{[v]} = P_i\otimes kv$$ because $\mathcal{O}_{\Pp^n}(-1)$ is the tautological bundle.
In this description, one can easily check that the map
\begin{align*}
    d_i(-i-1)|_{[v]}: P_i\otimes kv \to& P_{i+1}
\end{align*}
is given by
$$p\otimes v \mapsto vp$$
where we are invoking the $E$-module action.
In fact, this is the definition given in the original BGG paper \cite{Bernshtein1978AlgebraicBO}.
With this understanding of the differential map, we arrive at the following definition.
\begin{definition}
$P$ is faithful if for all $v\in V$, the following sequence of $k$-vector spaces 
$$\cdots\longrightarrow P_{i-1} \xrightarrow[]{\,\cdot v\,}P_i  \xrightarrow[]{\,\cdot v\,} P_{i+1}\longrightarrow \cdots$$ 
is exact at all $i < c$ 
where $P_i$ is the $i$th graded piece of $P$, $c = \max_i(P_i\neq 0)$, and $\cdot v$ denotes multiplication by $v$ due to the $E$-module structure.
\end{definition}
Because the differential map is given by $p\otimes v\mapsto vp$, faithfulness corresponds to exactness.
\begin{proposition} \label{prop:faithful_exact}
    $P$ is faithful if and only if the BGG sheaf is a vector bundle and $\tilde{L}(P)$ gives a resolution for it.
\end{proposition}

\subsection{Homological Dimension and Cohomology}
A key property of the homological dimension is that it can be inferred from the cohomologies of a bundle and its twists \cite[Proposition 1.4]{spindler_stability}.
\begin{proposition} \label{prop:hd1}
For a vector bundle $\mathcal{F}$, $\text{hd}(\mathcal{F}) \leq d$ if and only if $$\bigoplus_{i\in \mathbb{Z}}H^q(\Pp^n, \mathcal{F}(i)) = 0$$
for all $q$ such that $1\leq q \leq n-d-1$.
\end{proposition}
\begin{remark}
    By plugging in $d = n-1$, we see that $\text{hd}(\mathcal{F})\leq n-1$ is always true.
    This is why the homological dimension for a vector bundle on $\Pp^n$ can only be $0,1,...,$ or $n-1$.
\end{remark}
It is worth noting that throughout \cite{spindler_stability}, the authors assumed $k$ was characteristic 0. 
Nonetheless, the proof of Proposition \ref{prop:hd1} only involved performing a characteristic-blind induction with the base case being Horrock's splitting criterion. 
Given that Horrock's criterion holds in positive characteristics \cite[Theorem 3.1]{abe_splitting}, Proposition \ref{prop:hd1} also holds.

Similar to the module case, we call a resolution of sheaves 
$$
\cdots \longrightarrow \mathcal{F}^{-1} \longrightarrow \mathcal{F}^0 \longrightarrow 0
$$
a linear resolution if each $\mathcal{F}^i$ is a direct sum of copies of $\mathcal{O}_{\Pp^n}(i)$.
Once a linear resolution is given for a vector bundle, no shorter resolution can exist.
\begin{proposition}\label{prop:lin_res_hd}
If a vector bundle has a linear resolution of length $l$ where $l \leq n-1$, then its homological dimension is $l$.
\end{proposition}
\begin{proof}
Without loss of generality, $\mathcal{F}$ has a resolution
$$0 \longrightarrow \bigoplus \mathcal{O}_{\Pp^n}(-l)\longrightarrow \cdots\longrightarrow \bigoplus\mathcal{O}_{\Pp^n} \longrightarrow\mathcal{F} \longrightarrow 0$$
By definition, the homological dimension is less than or equal to $l$.
To show that it can be no less than $l$, we observe that $H^{n-l}(\Pp^n, \mathcal{F}(l-(n+1))) \neq 0$ by using the long exact sequence of cohomologies induced by the short exact sequence of sheaves inductively, so Proposition \ref{prop:hd1} gives us the result.
\end{proof}

\section{Lemma about Linear Resolutions} \label{sec:lemmas}
In this section, we prove a lemma that allows us to translate the simplicity of certain $E$-modules to the simplicity of their BGG-sheaves.
We will make use of the functor $\Gamma_*$ that turns sheaves into modules.
\begin{definition}
    The functor $\Gamma_*: \text{QC}(\Pp^n) \to S\text{Mod}^{\,gr}$ from quasi-coherent sheaves to graded $S$-modules is defined by $\Gamma_*(\mathcal{F}) = \bigoplus_{i\in \mathbb{Z}} H^0(\Pp^n, \mathcal{F}(i))$. The $i$th graded piece of $\Gamma_*(\mathcal{F})$ is defined to be $H^0(\Pp^n, \mathcal{F}(i))$.
\end{definition}
We will use $\tilde{M}$ or $M^\sim$ to denote the sheaf on $\Pp^n$ obtained by sheafifying a $S$-module $M$.
One property about Castelnuovo-Mumford regularity that will be crucial to us is the following \cite[Proposition 4.16]{Eisenbud_syzygies}.
\begin{proposition} \label{prop:regular_surjective}
    Let $M$ be a finitely generated graded $S$-module that is $d$-regular.
    Then the canonical map 
    $$
    M_d \to \Gamma_*(\tilde{M})_d
    $$
    is surjective.
\end{proposition}
In general, an exact sequence of sheaves does not produce an exact sequence of global sections.
However, knowing the exactness of $\tilde{L}(P)$ says a lot about the exactness of $L(P)$:
\begin{lemma} \label{lem: linear res}
Let $P = \bigoplus_{i=0}^c P_i$ be an $E$-module such that $c:=\max_i(P_i\neq 0)\in \{0,1,...,n\}$ and $H^i(\tilde{L}(P)) = 0$ for all $i\neq c$.
    Then, $L(P)$ is the minimal free resolution for $\Gamma_*(\mathcal{F})_{\geq -c}$, where $\mathcal{F}$ is the BGG-sheaf of $P$.
\end{lemma}

\begin{remark}
Although Lemma \ref{lem: linear res} is stated using the BGG correspondence, one can write it without the BGG language as follows.
    Let $$ \mathcal{C}^\bullet:\,\,\,\,\,
    0\longrightarrow \Oo_{\Pp^n}(-c)^{\oplus n_{-c}}\longrightarrow \Oo_{\Pp^n}(-c+1)^{\oplus n_{-c+1}} \longrightarrow\cdots 
    \longrightarrow\Oo_{\Pp^n}^{\oplus n_0} \longrightarrow 0
    $$
    be a resolution for some coherent sheaf $\mathcal{F}$ where $c\in \{0,1,...,n\}$.
    Then $\Gamma_*(\mathcal{C})$ is the minimal free resolution for the module
    $\Gamma_*(\mathcal{F})_{\geq 0}$.
\end{remark}

\begin{proof}[Proof of Lemma \ref{lem: linear res}]
Note that once we prove that $L(P)$ is a resolution, the fact that it is minimal is automatic since $L(P)$ is a linear free complex.
    We will prove the lemma by performing induction on $c$.
    The case for $c = 0$ is clear.
    Now, let $c = 1,...,n$ and assume the lemma is already true for $c-1$.
    In particular, the lemma applies for $P_{\leq c-1}$.
    To obtain the lemma for $c$, we will need to prove:
    \begin{itemize}
        \item $H^{c-1}(L(P)) = 0$, i.e. we have exact $S$-module maps at
        \begin{align} \label{eqn:exact_c-1}
            S[c-2]\otimes P_{c-2} \xrightarrow[]{\Gamma_*d_{c-2}} S[c-1]\otimes P_{c-1} \xrightarrow[]{\Gamma_*d_{c-1}} S[c] \otimes P_c
        \end{align}
        since the exactness for $0,1,...,c-2$ is gauranteed by the induction hypothesis;
        \item the exactness of
        \begin{align} \label{eqn:exact_c}
            S[c-1]\otimes P_{c-1} \xrightarrow[]{\Gamma_*d_{c-1}} S[c] \otimes P_c \to \Gamma_*(\mathcal{F})_{\geq -c}\to 0.
        \end{align}
    \end{itemize}
    We will first show the exactness of (\ref{eqn:exact_c-1}).
    We have the following commutative diagram
\[\begin{tikzcd}
	{P_{c-1}\otimes \mathcal{O}(c-1)} && {P_c\otimes\mathcal{O}(c)} \\
	& {\text{coker}\, d_{c-2}}
	\arrow["f", from=1-1, to=2-2]
	\arrow["\iota", from=2-2, to=1-3]
	\arrow["{d_{c-1}}"', from=1-1, to=1-3]
\end{tikzcd}\]
where $f$ is the natural cokernal map and $\iota$ is an inclusion of sheaves due to the exactness of $\tilde{L}(P)$ at $c-1$.
Note that we are using $\mathcal{O}(c)$ to denote $\mathcal{O}_{\Pp^n}(c)$ and the tensor product is over $ k$.
We then have 
\begin{align*}
    \ker \Gamma_*\, d_{c-1} = \ker \Gamma_*\iota \circ\Gamma_* f = \ker\, \Gamma_*f = \im \Gamma_* d_{c-2}
\end{align*}
where the first equality is due to the functoriality of $\Gamma_*$, the second equality is because $\Gamma_* \iota$ is an injection due to the left exactness of $\Gamma_*$, and the third equality is by applying the induction hypothesis on $P_{\leq c-1}$ because $\coker d_{c-2}$ is the BGG-sheaf of $P_{\leq c-1}$.
This shows the exactness of (\ref{eqn:exact_c-1}).

Next, we show the exactness of (\ref{eqn:exact_c}).
Let $M:= H^c(L(P))$.
Because the sheafification functor $(-)^\sim: M\mapsto \tilde{M}$ is exact, we know that $\tilde{M} \cong \mathcal{F}$.
Therefore, we have a natural map of graded $S$-modules 
$$\phi: M \to \Gamma_* \mathcal{F}.$$
It is obvious that $M = M_{\geq -c}$ because it is a quotient of $P_c \otimes S[c]$ so all the generators are of degree $-c$.
The exactness of (\ref{eqn:exact_c}) is then equivalent to $\phi$ being an isomorphism in degrees $\geq -c$, with the exactness at $S[c]\otimes P_c$ being equivalent to the injectivity of $\phi$ and the exactness at $(\Gamma_*\mathcal{F})_{\geq -c}$ equivalent to the surjectivity of $\phi$.

Let's show surjection of $\phi: M \to \Gamma_* \mathcal{F}_{\geq -c}$ first.
    $M$ has a linear resolution given by $L(P)$ and is generated by degree $-c$ elements, so it is $-c$-regular.
    Therefore, it is $d$-regular for any $d\geq -c$.
    Applying Proposition \ref{prop:regular_surjective}, we immediately obtain the surjection.
    
Next, we turn to showing injection for $\phi: M_{\geq -c+1} \to (\Gamma_* \mathcal{F})_{\geq -c+1}$ before proving injection in degree $-c$ as well.
We have the following exact sequence:
$$
    0\to \coker d_{c-2} \to P_c\otimes \mathcal{O}(c) \to \mathcal{F} \to 0.
    $$
    By the left exactness of $\Gamma_*$, we have injection $\iota: \Gamma_*(P_c\otimes \mathcal{O}(c)) /(\Gamma_* \coker d_{c-2}) \to \Gamma_* \mathcal{F}$.
    The following diagram clarifies the algebra happening.
\[\begin{tikzcd}
	&&&& {\Gamma_*\mathcal{F}} \\
	{\Gamma_*(P_{c-2}\otimes\mathcal{O}(c-2))} & {\Gamma_*(P_{c-1}\otimes\mathcal{O}(c-1))} && {\Gamma_*(P_{c}\otimes\mathcal{O}(c))} & M \\
	&& {\Gamma_*\text{coker }d_{c-2}} & {}
	\arrow["\subset", from=3-3, to=2-4]
	\arrow["{\Gamma_*d_{c-1}}"', from=2-2, to=2-4]
	\arrow[from=2-2, to=3-3]
	\arrow[from=2-4, to=2-5]
	\arrow[from=2-4, to=1-5]
	\arrow["\phi", from=2-5, to=1-5]
	\arrow["{\Gamma_*d_{c-2}}", from=2-1, to=2-2]
\end{tikzcd}\]
In the diagram, both the horizontal maps and the diagonal maps are exact at $\Gamma_*(P_c\otimes \mathcal{O}(c))$.
From the diagram, we see that $\im \Gamma_* d_{c-1} \subset \Gamma_*\coker d_{c-2}$.
Because $\coker d_{c-2}$ is the BGG-sheaf of $P_{\leq c-1}$, the induction hypothesis tells us
 that $$(\Gamma_* \coker d_{c-2})_{\geq -c+1} = \coker \Gamma_* d_{c-2} = \im \Gamma_* d_{c-1}.$$

The original map $\phi$ can be factored using $\phi = \iota \circ q$ where $q$ is the natural quotient as shown in the diagram below.
\[\begin{tikzcd}
	M & {\Gamma_*(P_c\otimes\mathcal{O}(c))/\text{im }\Gamma_* d_{c-1}} \\
	{\Gamma_* \mathcal{F}} & {\Gamma_*(P_c\otimes\mathcal{O}(c))/\Gamma_* \text{coker } d_{c-2}}
	\arrow["\cong", from=1-1, to=1-2]
	\arrow["\phi"', from=1-1, to=2-1]
	\arrow["q", from=1-2, to=2-2]
	\arrow["\iota", from=2-2, to=2-1]
\end{tikzcd}\]
Recall that $\iota$ is inclusion, so $\ker \phi = \ker q$.
Because $\im \Gamma_* d_{c-1}$ and $\Gamma_* \coker d_{c-2}$ are equal in degree $d\geq -c+1$, $\ker \phi$ can only be nonzero in degree $-c$. This proves injection for degrees larger than $c$.
    
Lastly, we show injection in degree $c$.
We have a short exact sequence
$$
0\to \ker \phi \to M \xrightarrow[]{\,\phi\,} (\Gamma_* \mathcal{F})_{\geq -c} \to 0.
$$
Because $\ker \phi$ is only in degree $-c$, we can construct a module map $M \to \ker \phi$ to obtain a splitting by the composition $$M \to M_{-c} \to \ker \phi$$ where $M\to M_{-c}$ is modding out all elements whose degrees are larger than $-c$ while $M_{-c} \to \ker\phi$ is just any projection of $k$-vector spaces.
Therefore, we obtain $$M \cong \ker\phi \oplus (\Gamma_* \mathcal{F})_{\geq -c}.$$
Exactness of (\ref{eqn:exact_c-1}) says that $ L(P)$ is the minimal free resolution for $M$.
Because $\ker\phi$ is just $\dim \ker\phi$ copies of $ k$, uniqueness of minimal free resolution says that $L(P)$ must contain copies of the Koszul complex if $\ker\phi\neq 0$.
However, $L(P)$ has length $c\leq n$, so the Koszul complex which has length $n+1$ is not a summand of $L(P)$.
This shows that $\phi$ is an isomorphism for degrees greater than or equal to $-c$, so (\ref{eqn:exact_c}) is indeed exact.
\end{proof}

This next corollary provides a direct way to prove the simplicity of the BGG-sheaf $\mathcal{F}$.
\begin{corollary} \label{cor:module_simple}
If  $\,\Hom_{E\text{Mod}^{\,\text{gr}}}(P,P) =  k$, then $\Hom_{\text{QC}(\Pp^n)}(\mathcal{F},\mathcal{F}) =  k$.
\end{corollary}
\begin{proof}
    Suppose for contradiction that there exists $\phi\in \Hom(\mathcal{F},\mathcal{F})$ such that $\phi\not\in k$.
    Because the isomorphism $\mathcal{F} \to (\Gamma_*\mathcal{F})^\sim$ is natural, $\Gamma_* \phi \not\in  k$ where $\Gamma_*\phi: \Gamma_*(\mathcal{F})_{\geq -c}\to \Gamma_*(\mathcal{F})_{\geq -c}$ is the induced map of modules (Chapter 2 Proposition 5.15 of \cite{Hartshorne}).
    $\Gamma_*\phi$ then lifts to a map $\phi'\in \Hom(L(P), L(P))$ of free resolutions because free modules are projective.
    Note that $\phi' \not\in  k$ because $\Gamma_*\phi\not\in  k$.
    Due to the equivalence of category between linear free complexes of $S$-modules and graded $E$-modules (Proposition \ref{prop:equiv_cat}), $\phi'$ is a nontrivial morphism in $\Hom(P,P)$ so there is a contradiction.
\end{proof}

\section{Construction of Simple Bundles}\label{sec:construction}
In this section we give the core construction needed to establish Theorem \ref{thm:arb_hd}.
From Corollary \ref{cor:module_simple} and Proposition \ref{prop:lin_res_hd}, we see that the problem of constructing simple bundles of homological dimension $l$ relates to the problem of constructing simple faithful $E$-modules with $l$ nonzero graded pieces.
To construct simple modules, we introduce a convenient definition.

\begin{definition} \label{def:anchor}
Let $U,W$ be vector spaces.
Suppose that $L\subset U\otimes W$ satisfy the property that for all $\phi \in \Hom(U, U)$ such that $(\phi\otimes 1)(L) \subset L$, it is the case that $\phi \in  k$ (i.e. $\phi$ is just a scaling).
Then we say $L$ \textbf{anchors} $U$.
\end{definition}
This notion is relevant for constructing simple modules due to the following.
\begin{lemma}\label{lem:simple_mod}
Let $P$ be the $E$-module defined by $P = P_0 \otimes_k\left(\bigoplus_{i = 0}^{l} \bigwedge^i V\right)$ where $P_0$ is a $k$-vector space.
Let $L \subset P_0 \otimes \bigwedge^l V$ be such that $L$ anchors $P_0$.
Then $\Hom_{E\text{Mod}^{\,\text{gr}}}(P / L, P/L) =  k$.
\end{lemma}
\begin{proof}
    Let $\phi: P/L \to P/L$ be a morphism of graded $E$-modules.
    Note that $\phi$ is completely determined by $\phi|_{P_0}$ because $P/L$ is generated by $P_0$ as an $E$-module.
    Let $\tilde{\phi}: P\to P$ be the morphism determined by $\phi|_{P_0}$.
    Because $\phi$ preserves $E$-scalar multiplication, we have the following commutative square
\[\begin{tikzcd}
	P & P \\
	{P/L} & {P/L}
	\arrow["q", from=1-1, to=2-1]
	\arrow["q", from=1-2, to=2-2]
	\arrow["{\tilde{\phi}}", from=1-1, to=1-2]
	\arrow["\phi", from=2-1, to=2-2]
\end{tikzcd}\]
    where $q$ is the natural quotient map.
    The commutivity implies that $$(\tilde{\phi}|_{P_0\otimes \wedge^l V})(L) = (\phi|_{P_0} \otimes 1_{\wedge^l V})(L) \subset L.$$
    Because $L$ anchors $P_0$, $\phi|_{P_0} \in  k$ so $\phi \in  k$.
\end{proof}
The above lemma tells us we can mod out by certain linear subspaces to construct simple $E$-modules.
For our purpose of producing vector bundles, we need the resulting $E$-module to be faithful.
Therefore, the following lemma is useful.
\begin{lemma} \label{lem:faithful_quotient}
Let $P$ be a faithful $E$-module and $l = \max_i{P_i\neq 0}$.
Let $$\chi_i:= \sum_{j=0}^i (-1)^{j-i}\dim P_j.$$
Let $k$ be such that $0 \leq k \leq \chi_l -n$ (we assume that $\chi_l\geq n$ here).
Then a general $L \in \Gr(k, P_l)$ has the property that $P/L$ is a faithful module.
\end{lemma}
\begin{proof}
We have the bilinear map: $b: P_{l-1} \times V \to P_l$ defined by the $E$-action.
For any $L \subset P_l$, $L \cap \im b = 0$ sufficiently shows that $P/L$ is faithful: because $P$ is already faithful, we just need to check the exactness of
$$P_{l-2} \xrightarrow[]{\,\cdot v\,}P_{l-1}  \xrightarrow[]{\,\cdot v\,} P_{l}/L$$
at $P_{l-1}$, which is gauranteed if $L \cap \im b = 0$.

As a map of algebraic varieties, a general fiber $F$ of $b$ has $\dim F \geq \chi_{l-2} + 1$.
This is because the map $\cdot v: P_{l-1} \to P_l$ has a kernel of dimension $\dim \ker(\cdot v) = \chi_{l-2}$ and there is the redundancy of relative scaling between $P_{l-1}$ and $V$.
Thus,
\begin{align*}
    \dim \overline{\im b} \leq& \dim P_l + n+ 1 - \chi_{l-2} - 1\\
    =&\chi_{l-1} + n
\end{align*}
where $\overline{\im b}$ denotes Zariski closure.
Note that $\im b \subset P_l$ being invariant under scaling implies $\overline{\im b}$ is too, so $\overline{\im b}$ can be viewed as the affine cone of some projective variety.
Then a general linear subspace of dimension less than $\dim P_l - \dim \im b \leq \chi_l - n$ will be disjoint from $\im b$ \cite[p.224]{harris1992}.
\end{proof}

The key linear algebra lemma that allows us to find anchoring linear subspaces is the following:
\begin{lemma}\label{lem:general_anchor}
    Let $n = \dim U>1$, $m = \dim W \geq 4$ where $U,W$ are finite dimensional $k$-vector spaces. 
    Let $d$ be an integer in the range $d\in (\frac{2n}{m}, mn-\frac{2n}{m})$.
    Then a general $d$-dimensional linear subspace $L \in \Gr(d, U\otimes W)$ anchors $U$.
\end{lemma}
Note that this bound is not sharp but is sufficient for our purpose of constructing vector bundles.
Before we prove Lemma \ref{lem:general_anchor}, we present the construction for simple bundles of all homological dimensions.

\begin{proof}[Proof of Theorem \ref{thm:arb_hd}]
Fix a $l \in\{ 1,2,...,n-1\}$ and $r \geq n$.
Let $p>0$ be an integer so that $$r< p\left({n\choose l} - \frac{2}{{n+1\choose l}}\right).$$
Let $P_0:=k^p$ and $$P:=P_0 \otimes \left(\bigoplus_{i=0}^l \bigwedge^i V\right).$$
Because $\bigoplus_{i = 0}^{l} \bigwedge^i V$ is faithful, $P$ is also faithful.
Simple calculation shows that $\chi_l = p{n\choose l}$.
By applying Lemma \ref{lem:general_anchor} and Lemma \ref{lem:faithful_quotient}, there exists a linear subspace $L \subset P_l = P_0 \otimes \bigwedge^l V$ with the following properties:
\begin{enumerate}
    \item $L$ anchors $P_0$,
    \item $P/L$ is a faithful module,
    \item $\dim L = p{n \choose l} - r$
\end{enumerate}
because $r\in \left[n, p\left({n\choose l} - \frac{2}{{n+1\choose l}}\right) \right)$ implies that 
$$p{n \choose l} - r \in \left(\frac{2p}{{{n+1}\choose l}},\, p{n+1\choose l} - \frac{2p}{{n+1\choose l}}\right)$$ (using the identity ${n+1\choose l} - {n \choose l} = {n\choose l-1}$ allows us to see this).
Let $\mathcal{F}$ be the BGG sheaf of $P/L$.
Then we observe the following about $\mathcal{F}$:
\begin{itemize}
    \item $\mathcal{F}$ is a vector bundle because $P/L$ is faithful;
    \item $\mathcal{F}$ has rank $r$ because $\dim L = p{n \choose l} - r$;
    \item $\mathcal{F}$ is simple because Lemma \ref{lem:simple_mod} tells us $P/L$ is simple and because of Proposition \ref{cor:module_simple};
    \item $\mathcal{F}$ has homological dimension $l$ because of Proposition \ref{prop:lin_res_hd}.
\end{itemize}
Therefore, we have constructed a vector bundle $\mathcal{F}$ of rank $r$ and homological dimension $l$. 
\end{proof}

\begin{example}
    We construct a simple bundle $\mathcal{F}$ of rank $5$ and homological dimension $2$ on $\Pp^3$ explicitly to illustrate the idea.
    Following the procedure described above, we consider the module
    $$
    P = k^2 \otimes \left(\bigoplus_{i = 0}^{2} \bigwedge^i V\right)
    $$
    and the quotient $P/L$ where $L$ is a general $1$-dimensional subspace of $k^2 \otimes \left(\bigwedge^2 V\right)$.
    Lemma \ref{lem:faithful_quotient} and Lemma \ref{lem:general_anchor} then tells us that $P/L$ is faithful and simple.
    Thus, $\tilde{L}(P/L)$ provides a resolution for its BGG-sheaf $\mathcal{F}$, which is of rank $5$ and homological dimension $2$:
    $$
    0\longrightarrow k^2 \otimes \mathcal{O}
    \longrightarrow (k^2\otimes V) \otimes \mathcal{O}(1)
    \longrightarrow
    \left(\left(k^2\otimes \bigwedge^2 V\right) / L\right) \otimes\mathcal{O}(2)
    \longrightarrow \mathcal{F}\longrightarrow 0.
    $$
    
\end{example}

The rest of this section is devoted to proving Lemma \ref{lem:general_anchor}.
\begin{proposition}\label{prop:anchor_open}
    Let $U,W$ be finite dimensional vector spaces with dimension $n,m$ respectively.
    The subset of $d$ dimensional subspaces that anchor $U$ is Zariski open in $\Gr(d, U\otimes W)$.
\end{proposition}
\begin{proof}
    We can see this by setting up an incidence correspondence and applying upper semi-continuity.
    Consider $$X\subset \Pp(\End(U\otimes W)) \times \Gr(d, U\otimes W)$$ such that $(\phi, L) \in X$ if and only if $\phi(L)\subset L$.
    Let $C\subset \Gr(d, U\otimes W)$ be the usual chart consisted of $d \times nm$ matrices whose first $d$ by $d$ minor is the identity.
    Then over $C$, $X$ is given by the equation $$\phi(L_i) \wedge L_1\wedge L_2\wedge ...\wedge L_d = 0$$
    for $i = 1,...,d$ where $L_i$ is the vector represented by the $i$th row of the $d \times nm$ matrix corresponding to subspace $L$ in chart $C$.
    Thus, $X$ is closed in $\Pp(\End(U\otimes W)) \times C$.
    Because these charts cover all of the Grassmannian, $X$ is a closed subset of $\Pp(\End(U\otimes W)) \times \Gr(d, U\otimes W)$.
    Consider $Z = \pi_1^{-1}(\Pp(\End(U)))\subset X$ where we are viewing $\Pp(\End(U)) \subset \Pp(\End(U\otimes W))$ as a linear subvariety using the identification $\phi \mapsto \phi \otimes 1$.
    Since $\pi_2: Z \to \Gr(d, U\otimes W)$ is projective, we can apply the upper semi-continuity of fiber dimensions (Corollary 13.1.5 of \cite{EGA}).
    $L$ anchoring $U$ is equivalent to the fiber $Z_L$ being zero dimensional so upper semi-continuity gives us the proposition.
\end{proof}
Note that Proposition \ref{prop:anchor_open} does not say anything about the subset being nonempty, which is needed to give us Lemma \ref{lem:general_anchor}.
\begin{proposition} \label{prop:anchor_duality}
    Let $U,W$ be finite dimensional vector spaces with dimension $n,m$ respectively.
    A general $d$ dimensional subspace of $U\otimes W$ anchors $U$ if and only if a general $nm-d$ dimensional subspace anchors $U$.
\end{proposition}
\begin{proof}
    We can define $$X^*\subset \Pp(\End(U^*\otimes W^*)) \times \Gr(nm-d, U^*\otimes W^*)$$ such that $(\psi, N) \in X^*$ if and only if $\psi(N)\subset N$.
    where $U^*$ is the dual space of $U$.
    We define $Z^*$ similarly using $\Pp(\End(U^*)) \subset \Pp(\End(U^*\otimes W^*))$.
    Let $f: Z \to Z^*$ be the map given by $(\phi,L) \mapsto (\phi^*, N)$ whre $\phi^*$ is the transpose of $\phi$ and $N$ is the space of linear forms that vanish on $L$.
    It is clear that $f$ is an isomorphism and the proposition follows. 
\end{proof}

We have a surjective rational map $$(U\otimes W)^{\oplus d} \cong U\otimes W \otimes k^d \cong k^{nmd} \dashrightarrow \Gr(d, U\otimes W)$$ defined by $(v_1,...,v_d) \mapsto v_1\wedge...\wedge v_d$.
After fixing bases for $U$ and $W$, we can write these data using indices conveniently.
Let $\mu,\nu$ be the indices for $k^d$, $i,j$ for $U$ and $a,b$ for $W$.
We can express $v_\mu$ by $$v_\mu = v_{i\mu a} u^i\otimes w^a$$ where $u^i, w^a$ are the basis vectors for $U,W$ respectively and Einstein summation is used.
Let $A \in \End(U)$.
Let $A^i_j$ be defined by $A(u^i) = A^i_j u^j$.
Given $A \in \End(U)$, the condition that $A$ fixes the subspace $$A(\langle v_1,...,v_d\rangle) \subset \langle v_1,...,v_d\rangle$$ is equivalent to the existence of matrix $C^\nu_{\mu}$ such that 
$A(v_\mu) = v_\nu C^\nu_\mu$.
Using basis and indices, this is $A^i_j v_{i\mu a}u^j\otimes w^a = v_{i\nu a}C^\nu_\mu u^i \otimes w^a$, which can be written as

\begin{equation} \label{eqn:tensor}
    A^j_i v_{j\mu a} = v_{i\nu a}C^\nu_\mu
\end{equation}
where it is understood that the equation applies for any $i = 1,...,n$, $\mu = 1,...,d$ and $a = 1,...,m$.
To summarize, we have the following:
\begin{proposition} \label{prop:tensor_solution}
Let $U,W$ be finite dimensional vector spaces with dimension $n,m$ respectively.
    If there exists an element $v_{j\mu a} \in U\otimes k^d \otimes W$ such that any pair of matrices $(A^j_i,C^\nu_\mu)$ that satisfies (\ref{eqn:tensor}) must be proportional to $(I_{n\times n}, I_{d\times d})$, then there exists an element $L\in \Gr(d, U\otimes W)$ that anchors $U$.
\end{proposition}
\begin{proof}
    Let $v_{j\mu a}$ satisfy the condition in the proposition.
    If the $v_1,...,v_d$ constructed from $v_{j\mu a}$ are linearly independent, then $L$ generated by them is an anchor because, as explained above, Equation \ref{eqn:tensor} is just a rewriting of the condition for being an invariant subspace.
    However, if $v_1,...,v_d$ are linearly dependent, we can find nonzero $K^\nu_\mu$ such that $v_{i\nu a}K^\nu_\mu = 0$, so $(I_{n\times n}, (I_{d\times d} + K))$ is an alternative solution for Equation \ref{eqn:tensor}. 
    Therefore, $v_1,...,v_d$ must be linearly independent.
    Thus, $L = \langle v_1,...,v_d\rangle$ will always anchor $U$.
\end{proof}

To construct the tensor $v_{i\mu a}$ that satisfy Proposition \ref{prop:tensor_solution}, we will use the following fact from linear algebra.
\begin{proposition} \label{prop:burnside}
    There exist $n$ by $n$ matrices $B_1,B_2\in M_{n\times n}$ such that any $C\in M_{n\times n}$ that commutes with both $B_1,B_2$ must be proportional to the identity.
\end{proposition}
\begin{proof}
    Let $B_1$ be a diagonal matrix with $n$ different values on the diagonal.
Then there are exactly $2^n$ subspaces of $k^n$ that are invariant under $B_1$ given by the direct sums of the $n$ distinct eigenspaces (p.811 of \cite{fillmore_invariant}).
Consider $n$ linearly independent vectors $x_1,...,x_n$.
Fix a subspace $U\subset k^n$ that is not $0$ or $k^n$.
The condition that no subset of $\{x_1,...,x_n\}$ forms a basis for $U$
is clearly a nonempty open condition for $\{x_1,...,x_n\}$ because we can express it using the wedge product $\wedge$ and the Plücker embedding.
Thus, the condition that no subset of $\{x_1,...,x_n\}$ forms a basis for any of the $2^n-2$ nontrivial invariant subspaces of $B_1$ is an nonempty open condition on the space of $n$ linearly independent vectors.
Pick $x_1,...,x_n$ that satisfy this condition.
Let $B_2$ be the matrix that has distinct eigenvalues for each of $x_1,...,x_n$.
Then, by construction, the only invariant subspaces shared by $B_1,B_2$ are $0$ and $k^n$.
Burnside's theorem on matrix algebra then says that $B_1,B_2$ generate $\End(k^n)$ \cite{burnside}.
Thus, any $C$ that commutes with $B_1,B_2$ must commute with all matrices, so it has to be proportional to the identity.
\end{proof}

Now we come to the key construction that will allow us to prove the subset of anchoring subspaces is nonempty.
\begin{proposition}\label{prop:matrix_bound}
    Given $m\geq \max(\frac{n}{d} ,\frac{d}{n}) + 2$, there exists $v_{i\mu a}$ such that the only solutions for Equation \ref{eqn:tensor} are proportional to $(I_{n\times n}, I_{d\times d})$
\end{proposition}
\begin{proof}
    Equation \ref{eqn:tensor} is a set of $m$ different matrix equations:
    \begin{align*}
        \{A V_a = V_a C\}_{a=1,...,m}
    \end{align*}
    where $A, V_a, C$ are $n\times n,$ $ n\times d$, and $d\times d$ respectively.
    Because taking transpose is allowed, the roles of $A$ and $C$ are interchangeable as far this proposition is concerned.
    Thus, without loss of generality, we assume that $n\geq d$.
    In this case, if we prove the proposition for $m = \lceil \frac{n}{d}\rceil + 2$, any larger $m$ will also work since they just add more constraints.
    Given a matrix $M$, let $M[r_1: r_2| c_1:c_2]$ denote the $(r_2 - r_1)\times (c_2 - c_1)$ submatrix starting at row $r_1$ and column $c_1$.
    
    For $a = 1,...,\lceil \frac{n}{d}\rceil$, let $V_a$ be defined by 
    $$
    V_a[(a-1)d + 1: \min((a-1)d + 1 + d, n + 1)| 1:d+1] = I_{d\times d}[1: \min(d, n-(a-1)d) + 1| 1: d+1]
    $$
    and zero everywhere else.
    Let $V_{m-1}, V_{m}$ be be any $n\times d$ matrices such that
    \begin{align*}
        V_{m-1}[1: d+1| 1:d+1] = B_1\\
        V_{m}[1: d+1| 1:d+1] = B_2
    \end{align*}
    where $B_i$s are the $d\times d$ matrices appearing in Proposition \ref{prop:burnside}.
    Given that $AV_a = V_aC$ holds for $a = 1,...,\lceil \frac{n}{d}\rceil$, due to $V_a$ being the identity for specific rows and $0$ elsewhere, we see that the matrix $A$ must satisfy the following constraints:
    \begin{itemize}
        \item for $a = 1,...,\lceil \frac{n}{d}\rceil$
        \begin{align} \label{eqn:matrix}
            &A[(a-1)d + 1: \min((a-1)d + 1 + d, n + 1)| (a-1)d + 1: \min((a-1)d + 1 + d, n + 1)] \\= &C[1: \min(d, n-(a-1)d) + 1| 1, 1: \min(d, n-(a-1)d) + 1]
        \end{align}
        \item The other entries of $A$ must be identically zero.
    \end{itemize}
    Given that $AV_a = V_aC$ holds for $a = m-1, m$, we see that $CB_i = B_iC$ for $i = 1,2$.
    Thus, $C = \lambda I_{d\times d}$ for some $\lambda \in  k$, so by Equation \ref{eqn:matrix}, we know $A = \lambda I_{n\times n}$.
\end{proof}

Now, we are finally ready for a proof of Lemma \ref{lem:general_anchor}.
\begin{proof}[Proof of Lemma \ref{lem:general_anchor}]
Let us check that for $d \in (\frac{2n}{m}, \lfloor\frac{nm}{2}\rfloor]$, the condition of Proposition \ref{prop:matrix_bound}
$$
m\geq \max(\lceil\frac{n}{d}\rceil,\lceil\frac{d}{n}\rceil) + 2
$$
is met. 
For $d\in (\frac{2n}{m},n]$, we have
\begin{align*}
    \max(\lceil\frac{n}{d}\rceil,\lceil\frac{d}{n}\rceil) + 2 = \lceil\frac{n}{d}\rceil + 2
    \leq \lceil\frac{nm}{2n}\rceil + 2
    =\lceil\frac{m}{2}\rceil + 2
    \leq m
\end{align*}
when $m\geq 4$.
For $d\in (n,\lfloor\frac{nm}{2}\rfloor]$, we have
\begin{align*}
    \max(\lceil\frac{n}{d}\rceil,\lceil\frac{d}{n}\rceil) + 2 = \lceil\frac{d}{n}\rceil + 2
    \leq \lceil\frac{nm}{2n}\rceil + 2
    =\lceil\frac{m}{2}\rceil + 2
    \leq m
\end{align*}
when $m\geq 4$. Thus, by Proposition \ref{prop:tensor_solution} and Proposition \ref{prop:anchor_open}, for these values of $d$, a general $d$-dimensional subspace of $U\otimes W$ will anchor $U$.
For $d\in (\lfloor\frac{nm}{2}\rfloor, nm-\frac{2n}{m})$, we simply apply Proposition \ref{prop:anchor_duality}, considering the case when $nm$ is odd and the case when $nm$ is even separately.
\end{proof}

\section*{Acknowledgements}\label{sec:acknowledgements}
First and foremost, I have to thank Professor Mihnea Popa for supervising my research. 
This project would not have been possible without him.
Secondly, I must thank Wanchun Shen and Duc Vo. 
They helped me understand the key literature in the area by guiding me in the Directed Reading Program and provided valuable insights when I discussed ideas with them.
Lastly, I have to thank Professor Marcos Jardim for answering my questions about his paper, which directly inspired my research \cite{Jardim}.

\clearpage


\begin{thebibliography}{10}

\bibitem{abe_splitting}
Takuro Abe and Masahiko Yoshinaga.
\newblock Splitting criterion for reflexive sheaves.
\newblock {\em Proc. Amer. Math. Soc.}, 136(6):1887--1891, 2008.

\bibitem{Bernshtein1978AlgebraicBO}
I.~N. Bernshtein, Izrail~Moiseevich Gel'fand, and S.~I. Gel'fand.
\newblock Algebraic bundles over $\mathbb{P}^n$ and problems of linear algebra.
\newblock {\em Functional Analysis and Its Applications}, 12:212--214, 1978.

\bibitem{spindler_stability}
Guntram Bohnhorst and Heinz Spindler.
\newblock The stability of certain vector bundles on $\mathbb{P}^n$.
\newblock In Klaus Hulek, Thomas Peternell, Michael Schneider, and Frank-Olaf
  Schreyer, editors, {\em Complex Algebraic Varieties}, pages 39--50, Berlin,
  Heidelberg, 1992. Springer Berlin Heidelberg.

\bibitem{brambilla2007cokernel}
Maria~Chiara Brambilla.
\newblock Cokernel bundles and {F}ibonacci bundles.
\newblock {\em Math. Nachr.}, 281(4):499--516, 2008.

\bibitem{fillmore_invariant}
L.~Brickman and P.~A. Fillmore.
\newblock The invariant subspace lattice of a linear transformation.
\newblock {\em Canadian J. Math.}, 19:810--822, 1967.

\bibitem{Coanda}
I.~Coandă and G.~Trautmann.
\newblock Horrocks theory and the {B}ernstein-{G}el'fand-{G}el'fand
  correspondence.
\newblock {\em Transactions of the American Mathematical Society},
  358(3):1015--1031, 2006.

\bibitem{EGA}
Jean Dieudonn{\'e} and Alexander Grothendieck.
\newblock \'{E}l\'ements de g\'eom\'etrie alg\'ebrique.
\newblock {\em Inst. Hautes \'Etudes Sci. Publ. Math.}, 4, 1961--1967.

\bibitem{Eisenbud_syzygies}
David Eisenbud.
\newblock {\em The geometry of syzygies}, volume 229 of {\em Graduate Texts in
  Mathematics}.
\newblock Springer-Verlag, New York, 2005.
\newblock A second course in commutative algebra and algebraic geometry.

\bibitem{Eisenbud_floystad}
David Eisenbud, Gunnar Fl{\o}ystad, and Frank-Olaf Schreyer.
\newblock Sheaf cohomology and free resolutions over exterior algebras.
\newblock {\em Transactions of the American Mathematical Society},
  355(11):4397--4426, jul 2003.

\bibitem{burnside}
Israel Halperin and Peter Rosenthal.
\newblock Burnside's theorem on algebras of matrices.
\newblock {\em American Mathematical Monthly}, 87, 12 1980.

\bibitem{harris1992}
Joe Harris.
\newblock {\em Algebraic geometry}, volume 133 of {\em Graduate Texts in
  Mathematics}.
\newblock Springer-Verlag, New York, 1992.
\newblock A first course.

\bibitem{Hartshorne}
Robin Hartshorne.
\newblock {\em Algebraic geometry}, volume No. 52 of {\em Graduate Texts in
  Mathematics}.
\newblock Springer-Verlag, New York-Heidelberg, 1977.

\bibitem{Hartshorne_conjecture}
Robin Hartshorne.
\newblock Algebraic vector bundles on projective spaces: A problem list.
\newblock {\em Topology}, 18(2):117--128, 1979.

\bibitem{Jardim}
Marcos Jardim and Daniela~Moura Prata.
\newblock Pure resolutions of vector bundles on complex projective spaces.
\newblock preprint arXiv:1210.7835.

\bibitem{Popa_bgg}
Robert Lazarsfeld and Mihnea Popa.
\newblock Derivative complex, {BGG} correspondence, and numerical inequalities
  for compact {K}\"{a}hler manifolds.
\newblock {\em Invent. Math.}, 182(3):605--633, 2010.

\end{thebibliography}

\newpage

\end{document}